\documentclass[10pt, reqno]{amsart}
\usepackage{amscd, amssymb, amsmath}
\usepackage{hyperref}
\usepackage{epsfig}
 \usepackage{verbatim}
\usepackage{dsfont}

\newtheorem{Df}{Definition}
\newtheorem{theorem}[Df]{Theorem}

\newtheorem{lemma}[Df]{Lemma}

\newtheorem{corollary}[Df]{Corollary}
\numberwithin{Df}{subsection}
\numberwithin{equation}{subsection}

\newcommand{\coev}{\ensuremath{\operatorname{coev}} }
\newcommand{\ev}{\ensuremath{\operatorname{ev}} }
\newcommand{\tev}{\widetilde{\operatorname{ev}}}
\newcommand{\tcoev}{\widetilde{\operatorname{coev}}}

\newcommand{\Pu}{\ensuremath{{P_\unit}}}

\newcommand{\C}{\ensuremath{\mathbb{C}}}
\newcommand{\Z}{\ensuremath{\mathbb{Z}}}

\newcommand{\gl}{\ensuremath{\mathfrak{gl}(m|n)}}

\newcommand{\g}{\ensuremath{\mathfrak{g}}}

\newcommand{\cat}{\mathcal{C}}
\newcommand{\ideal}{I}

\newcommand{\ob}{Ob(\mathcal{C})}

\newcommand{\End}{\operatorname{End}}
\newcommand{\Hom}{\operatorname{Hom}}

\newcommand{\Id}{\operatorname{Id}}

\newcommand{\proj}{\ensuremath{\operatorname{e}} }
\newcommand{\kt}{$\Bbbk$\nobreakdash-\hspace{0pt}}
\newcommand{\kk}{\Bbbk}
\newcommand{\mt}{\operatorname{\mathsf{t}}}

\newcommand{\Fcat}{\mathcal{F}}
\newcommand{\Dcat}{\mathcal{D}}

\newcommand{\Proj}{\ensuremath{\mathcal{P}roj}}
\newcommand{\fg}{\ensuremath{\mathfrak{g}}}
\newcommand{\fq}{\ensuremath{\mathfrak{q}}}
\newcommand{\fh}{\ensuremath{\mathfrak{h}}}
\newcommand{\fb}{\ensuremath{\mathfrak{b}}}
\newcommand{\0}{\ensuremath{\bar{0}}}
\newcommand{\1}{\ensuremath{\bar{1}}}
\newcommand{\unit}{\ensuremath{\mathds{1}}}
\newcommand{\St}{\ensuremath{\operatorname{St}}}
\newcommand{\FK}{\kk}

\begin{document}
\title[Ambidextrous objects and nontrivial trace functions]{Ambidextrous objects and trace functions for nonsemisimple categories}

\author{Nathan Geer}
\address{Mathematics \& Statistics\\
  Utah State University \\
  Logan, Utah 84322, USA}
\thanks{Research of the first author was partially supported by NSF grants
 DMS-0968279 and DMS-1007197.}\
\email{nathan.geer@usu.edu}
\author{Jonathan Kujawa}
\address{Mathematics Department\\
University of Oklahoma\\
Norman, OK 73019}
\thanks{Research of the second author was partially supported by NSF grant
DMS-0734226 and NSA grant H98230-11-1-0127.}\
\email{kujawa@math.ou.edu}
\author{Bertrand Patureau-Mirand}
\address{LMAM, universit\'e de Bretagne-Sud, universit\'e europ\'eenne de
  Bretagne, BP 573, 56017 Vannes, France }
\email{bertrand.patureau@univ-ubs.fr}
\date{\today}

\begin{abstract} 
  We provide a necessary and sufficient condition for a simple object in a
  pivotal \kt category to be ambidextrous. In turn, these objects imply the existence of nontrivial trace functions in the category.  These functions play an important role in low-dimensional topology as well as in studying the category itself.   In particular, we prove
  they exist for factorizable ribbon Hopf algebras, modular representations of
  finite groups and their quantum doubles, complex and modular Lie
  (super)algebras, the $(1,p)$ minimal model in conformal field theory, and
  quantum groups at a root of unity.
 \end{abstract}

\maketitle
\setcounter{tocdepth}{1}

\section{Introduction}

\subsection{} Let $\cat$ be a category with a tensor product and duality.  Assuming $\cat$ satisfies certain minimal axioms, then one can use the tensor product and duality structure on $\cat$ to define (categorical) traces of endomorphisms and (categorical) dimensions of objects.   These trace and dimension functions are powerful tools for studying $\cat$.  The existence of trace and dimension functions also allows one to use $\cat$ to construct invariants of knots, links, 3-manifolds, and other objects in low-dimensional topology (e.g.\  see \cite{Tu}). 

However, in many situations of great interest the category is not semisimple and these functions are trivial.  This occurs, for example, for the typical simple supermodules of a complex Lie superalgebra and for certain representations of a quantum group at a root of unity.  It is desirable to find a suitable replacement for trace and dimension in these settings.  Addressing this question, the first and third authors of this paper showed by direct calculation that the typical simple supermodules for the Lie superalgebras of type A and C admit modified trace and dimension functions and that these in turn give rise to nontrivial link invariants \cite{GP2}.  Geer and Patureau-Mirand have since worked with a number of coauthors to obtain further examples of modified traces.  They have successfully used these functions to define invariants of knots, links, 3-manifolds, and other low-dimensional objects in topology.

Most recently, in \cite{GKT} and \cite{GPV} this approach is used to
understand and generalize the ``state sum'' invariants of 3-manifolds
introduced by Turaev-Viro and Kashaev.  As an outcome of their investigations
they obtain
3-manifold invariants, ``relative Homotopy Quantum Field Theories'',
and a generalization of Kashaev's quantum dilogarithm invariant which 
was introduced in his foundational paper where he first stated the 
volume conjecture.   
 
In part motivated by these 
 developments in quantum topology, the authors of the present paper unified and generalized the notion of modified trace and dimension functions to the setting of ribbon categories in \cite{GKP}.  In particular, we showed that these functions also give unexpected new insights into purely representation theoretic questions.  For example, they provide a natural generalization of the well known conjecture of Kac and Wakimoto on the superdimension of simple supermodules for complex Lie superalgebras \cite{KW}.   Recently, Serganova proved our generalized Kac-Wakimoto conjecture for $\gl$ along with the ordinary Kac-Wakimoto conjecture for the basic classical Lie superalgebras \cite{serganova4}.  In turn, the generalized Kac-Wakimoto conjecture is a key ingredient in the forthcoming calculation of complexity for the simple $\gl$-supermodules \cite{BKN4}.

A simple object admits a nontrivial trace if and only if it is ambidextrous (a condition on morphisms).  One surprising outcome of \cite{GKP} was the discovery that ambidextrous objects can fail to exist even in a natural setting like the finite-dimensional representations of $\mathfrak{sl}_{2}(\kk)$ over a field of positive characteristic.

The main result of this paper is the reformulation of ambidextrous into a condition on objects.  This provides a new perspective on what it means for an object to be ambidextrous.   Furthermore, the new condition can easily be verified in a wide variety of settings.  As a consequence we recover results computed in \cite{GKP}, \cite{GP}, and \cite{GP2} as well as a large number of previously unknown examples.  A striking outcome of the main theorem and the subsequent examples is the fact that ambidextrous objects seem to be quite plentiful in nature. In particular, we show they exist for representations of factorizable ribbon Hopf algebras, finite groups and their quantum doubles, Lie (super)algebras, the $(1,p)$ minimal model in conformal field theory, and quantum groups at a root of unity.
\subsection{Acknowledgements}  The second author is grateful to David Hemmer, Christopher Drupieski, and Christopher Bendel for helpful conversations.

\section{Traces on Pivotal  $\kk$-categories}

\subsection{Pivotal categories} 
We recall the definition of a pivotal tensor category, see for instance,
\cite{BW}. A \emph{tensor category} $\cat$ is a category equipped with a
covariant bifunctor $\otimes :\cat \times \cat\rightarrow \cat$ called the
tensor product, an associativity constraint, a unit object $\unit$, and left
and right unit constraints such that the Triangle and Pentagon Axioms hold.
When the associativity constraint and the left and right unit constraints are
all identities we say that $\cat$ is a \emph{strict} tensor category. By
MacLane's coherence theorem, any tensor category is equivalent to a strict
tensor category. To simplify the exposition, we formulate further definitions
only for strict tensor categories; the reader will easily extend them to
arbitrary tensor categories.  In what follows we adopt the convention that $fg$ will denote the composition of morphisms $f \circ g$.

A strict tensor category $\cat$ has a \emph{left duality} if for each object
$V$ of $\cat$ there is an object $V^*$ of $\cat$ and morphisms
\begin{equation}\label{lele}
  \coev_{V} : \:\:   \unit \rightarrow V\otimes V^{*} \quad {\rm {and}} \quad
   \ev_{V}: \:\:
  V^*\otimes V\rightarrow \unit
\end{equation}
such that
\begin{align*}
  (\Id_V\otimes \ev_V)(\coev_V \otimes \Id_V)&=\Id_V & & {\rm {and}} &
  (\ev_V\otimes \Id_{V^*})(\Id_{V^*}\otimes \coev_V)&=\Id_{V^*}.
\end{align*}
A left duality determines for every morphism $f:V\to W$ in $\cat$ the dual (or 
transpose) morphism $f^*:W^*\rightarrow V^*$ by
$$
f^*=(\ev_W \otimes \Id_{V^*})(\Id_{W^*} \otimes f \otimes
\Id_{V^*})(\Id_{W^*}\otimes \coev_V),
$$
and determines for any objects $V,W$ of $\cat$, an isomorphism
$\gamma_{V,W}: W^*\otimes V^* \rightarrow (V\otimes W)^*$ by
$$
\gamma_{V,W} = (\ev_W\otimes \Id_{(V\otimes W)^*})(\Id_{W^*} \otimes \ev_V \otimes
\Id_W \otimes \Id_{(V\otimes W)^*})(\Id_{W^*}\otimes \Id_{V^*} \otimes
\coev_{V\otimes W}).
$$

Similarly, $\cat$ has a \emph{right duality} if for each object $V$ of $\cat$
there is an object $ V^\bullet$ of $\cat$ and morphisms
\begin{equation}\label{roro}
  \tcoev_{V} : \:\:   \unit\rightarrow V^\bullet\otimes V \quad {\rm {and}} \quad
  \tev_{V}:\:\:   V\otimes V^\bullet\rightarrow \unit
\end{equation}
such that
\begin{align*}
  (\Id_{V^\bullet}\otimes \tev_V)(\tcoev_V \otimes \Id_{V^\bullet})&=\Id_{V^\bullet}
  & & {\rm {and}} & (\tev_V\otimes \Id_{V})(\Id_{V}\otimes \tcoev_V)&=\Id_{V}.
\end{align*}
The right duality determines for every morphism $f:V\to W$ in $\cat$ the dual
morphism $f^\bullet:W^\bullet\rightarrow V^\bullet$ by
$$
f^\bullet=(\Id_{V^\bullet} \otimes \tev_W ) (\Id_{V^\bullet} \otimes f \otimes
\Id_{W^\bullet})( \tcoev_V \otimes \Id_{W^\bullet}),
$$
and determines for any objects $V,W$, an isomorphism $\gamma'_{V,W}:
W^\bullet\otimes V^\bullet \rightarrow (V\otimes W)^\bullet$ by
$$
\gamma'_{V,W}
= ( \Id_{(V\otimes W)^\bullet} \otimes \tev_V )(\Id_{(V\otimes W)^\bullet}
\otimes \Id_{V} \otimes \tev_W \otimes \Id_{V^\bullet} )(\tcoev_{V\otimes W}\otimes
\Id_{W^\bullet}\otimes \Id_{V^\bullet}).
$$

A \emph{pivotal category} is a tensor category with left duality $\{{\coev_V},
\ev_V\}_V $ and right duality $\{{\tcoev_V}, \tev_V\}_V $ which are compatible
in the sense that $V^*=V^\bullet$, $f^*=f^\bullet$, and
$\gamma_{V,W}=\gamma'_{V,W}$ for all $V, W, f$ as above.  Every pivotal
category
gives a  natural tensor isomorphism
\begin{equation}\label{E:DefphiV}
\phi=\{\phi_V=(\tev_{V}\otimes\Id_{V^{**}})(\Id_V\otimes\coev_{V^{*}})\colon V\to
V^{**}\}_{V \in \cat}.
\end{equation}

\subsection{Ribbon categories} We now relate the above setup to  ribbon categories. 
 A \emph{braiding} on a tensor category $\cat$ consists of a family of isomorphisms $\{c_{V,W}: V \otimes W \rightarrow W\otimes V \}$
satisfying the Hexagon Axiom \cite[XIII.1 (1.3-1.4)]{Kas} and the naturality condition expressed in the commutative diagram \cite[(XIII.1.2)]{Kas}, where $V,W$ run over objects of $\cat$.
We say a tensor category is \emph{braided} if it has a braiding.  We call the braiding on a tensor category \emph{symmetric} if $c_{W, V}  c_{V,W} = \Id_{V \otimes W}$ for all $V$ and $W$ in $\cat$.

A \emph{twist} in a braided tensor category $\cat$ with  duality is a family $\{ \theta_{V}:V\rightarrow V \}$ of natural isomorphisms defined for each object $V$ of $\cat$ satisfying relations \cite[(XIV.3.1-3.2)]{Kas}.    A \emph{ribbon category} is a braided tensor category with left duality and a twist.   Note that a ribbon category has a natural right duality given by 
\begin{align*}
\tcoev_V &=(\Id_{V^*}\otimes \theta_V)  c_{V,V^*}  b_V, &
\tev_V & = d_V  c_{V,V^*}  (\theta_V\otimes \Id_{V^*}).
\end{align*}
Moreover, this right duality is compatible 
with the left duality and defines a pivotal structure (cf.\ \cite[Section 2.2]{BK}). 

\subsection{Tensor \kt categories}
Let $\kk$ be a commutative ring.
A \emph{tensor \kt category} is a tensor category $\cat$ such that its
 hom-sets are left \kt modules, the composition and tensor product of morphisms are \kt bilinear,
  and  $\End_\cat(\unit)$ is a free \kt module of rank one.
Then the map $\kk \to \End_\cat(\unit), k \mapsto k \, \Id_\unit$  is a
\kt algebra isomorphism. It is used to identify $\End_\cat(\unit)=\kk$.
An object $V$ of a tensor \kt category $\cat$ is
\emph{absolutely simple} if $\End_\cat(V)$ is a free \kt module of rank one.
Equivalently, $V$ is  absolutely simple if the \kt homomorphism $\kk \to
\End_\cat(X),\, k   \mapsto  k\, \Id_X$  is an isomorphism.  
 If $V$ is absolutely simple, it is used to identify $\End_\cat(V)=\kk$. By the
definition of a tensor \kt category, the unit object $\unit$ is
 absolutely simple.  We call an object $V$ of $\cat$ \emph{absolutely indecomposable} if 
\[
\End_{\cat}(V)/\operatorname{Rad}(\End_{\cat}(V)) \cong \FK.
\]

\subsection{Traces}  
 We now recall the notion of a trace on an ideal in a pivotal \kt category.  For more details see \cite{GPV}.   By a \emph{right ideal} of  $\cat$ we mean a full subcategory, $\ideal$, of $\cat$ such that:  
\begin{enumerate}
\item  If $V$ is an object of $\ideal$ and $W$ is any object of $\cat$, then $V\otimes W$ is an object of $\ideal$.
\item If $V$ is an object of $\ideal$, $W$ is any object of $\cat$, and there exists morphisms  $f:W\to V$,  $g:V\to W$ such that $g  f=\Id_W$, then $W$ is an object of $\ideal$.
\end{enumerate}

If $\ideal$ is a right ideal in a pivotal \kt category $\cat$ 
then a \emph{right trace} on $\ideal$ is a family of linear functions
$$\{\mt_V:\End_\cat(V)\rightarrow \FK \}$$
where $V$ runs over all objects of $\ideal$ and such that following two conditions hold.
\begin{enumerate}
\item  If $U\in \ideal$ and $W\in \ob$ then for any $f\in \End_\cat(U\otimes W)$ we have
\begin{equation*}
\mt_{U\otimes W}\left(f \right)=\mt_U \left( (\Id_U\otimes \tev_W)(f\otimes \Id_{W^*})(\Id_U\otimes \coev_W)   \right).
\end{equation*}
\item  If $U,V\in \ideal$ then for any morphisms $f:V\rightarrow U $ and $g:U\rightarrow V$  in $\cat$ we have 
\begin{equation*}
\mt_V(g f)=\mt_U(f  g).
\end{equation*} 
\end{enumerate}

\subsection{Ambidextrous objects} 
An  absolutely simple object $V$ in a pivotal \kt category $\cat$ is said to be \emph{right ambidextrous} if
\[
f\coev_V =(\phi_V^{-1} \otimes \Id_{V^*})f^*\tcoev_{V^*}
\]
for all $f \in \End_\cat(V \otimes V^*)$ where $\phi_V$ is given in \eqref{E:DefphiV}.  This definition is equivalent to several other definitions, see \cite[Lemma~9]{GPV}.  In particular, when $\cat$ is a ribbon \kt category the definition of a right ambidextrous object is equivalent to the definition of an ambidextrous object given in \cite{GKP}.  For short we say $V$ is \emph{right ambi} if $V$ is right ambidextrous. 

Let $\ideal_V$ be the full subcategory of all objects $U$ satisfying  the
property that there exists an object $W$ and morphisms $\alpha:U\rightarrow
V\otimes W$ and $\beta:V\otimes W \rightarrow U$ with $\beta \alpha=\Id_U$.
It is not difficult to verify $\ideal_{V}$ forms a right ideal.    Combining \cite[Lemma 9(b)]{GPV} and  \cite[Theorem 7(a)]{GPV}, if $V$ is a right ambi object, then the canonical map $\End_{\cat}(V) \to \FK $ extends uniquely to a right trace on $\ideal_{V}$.  In particular, since on $\End_{\cat}(V)$ this right trace coincides with the canonical map  it follows that it is necessarily nonzero.  In short we have the following result.

\begin{theorem}\label{T:idealtheorem}
  If $\cat$ is a pivotal \kt category and ${V}$ is a right ambi object in
  $\cat$, then there is a unique non-zero right trace on $\ideal_{V}$ up to
  multiplication by an element of $\kk$.
\end{theorem}

\subsection{Variations}  In a pivotal \kt category there are also natural notions of left and two-sided ideals, left and two-sided traces, and left ambidextrous objects.  See \cite{GPV} for details.  In this paper we only considered the right-handed version of these concepts and leave the other variants to the interested reader.   When the category is a ribbon category, the various notions coincide and are equivalent to the definitions given in \cite{GKP}.  In this case we drop the adjective ``right'' for brevity.

\section{Main Theorem}
\subsection{}\label{SS:MainTheorem}  For the reminder of the paper we assume $\kk$ is a field.  We also assume that $\cat$ is an additive pivotal \kt category where every indecomposable object $V$ in $\cat$ is absolutely indecomposable and all elements of the radical of $\End_{\cat}(V)$ are nilpotent.  In particular, $\End_{\cat}(V)$ is a local ring.   We remark that these assumptions are known to imply that the Krull-Schmidt Theorem holds in $\cat$.

For example, by Fitting's Lemma our assumptions hold whenever $\kk$ is algebraically closed and $\cat$ is an abelian \kt category with all objects having finite length.  For a nonabelian example, we note the conditions also hold for Deligne's category $\operatorname{Rep}(S_{t})$ \cite{CO}.  See \cite{CK} for a description of trace and dimension functions in $\operatorname{Rep}(S_{t})$.

Given an object $V$ in $\cat$, we fix a direct sum decomposition of $V\otimes V^*$ into indecomposable objects
$W_i$ indexed by a set $I$: 
\begin{equation}\label{E:Decomp}
V\otimes V^* = \bigoplus_{k\in I}W_k.
\end{equation}
We write $i_{k}: W_{k}\to V\otimes V^{*}$ and $p_{k}: V \otimes V^{*} \to
W_{k}$ for the 
byproduct morphisms corresponding to this decomposition.  In particular,
$p_{k}i_{k} = \Id_{W_{k}}$ for all $k \in I$ and if we set
\[
\proj_k = i_{k}p_{k}:V\otimes V^*\to V\otimes V^*,
\]
then $\{\proj_{k} \}_{k\in I}$ is a pairwise orthogonal set of idempotents in
$\End_{\cat}(V \otimes V^{*})$ which sum to the identity.

\begin{lemma}\label{L:indecomp}  
  Let $V$ be an absolutely simple object in $\cat$.  Then there is a unique $j
  \in I$ which satisfies the following equivalent conditions:
  \begin{enumerate}
  \item $  \proj_j  \coev_V= \coev_V$; 
  \item
    $\Hom_{\cat}(\unit , W_{j}) $ is non-zero and is spanned by $p_{j}
    \coev_{V}$.
  \end{enumerate}   
  There is also a unique $j' \in I$ which satisfies the following equivalent
  conditions:
  \begin{enumerate}
  \item $ \widetilde{\ev}_V \proj_{j'}=\widetilde{\ev}_V;$ 
  \item 
    $\Hom_{\cat}(W_{j'} , \unit)$ is non-zero and is spanned by
    $\widetilde{\ev}_{V} i_{j'}$.
  \end{enumerate}
\end{lemma}
\begin{proof} 
  Using that the elements $\proj_{j}$ are a pairwise orthogonal set of
  idempotents which sum to the identity and that $V$ is absolutely simple, it
  is straightforward to verify that there is a unique $j$ such that $ \proj_j
  \coev_V= \coev_V$.  Using additivity we have
  \[
  \Hom_{\cat}(\unit , V \otimes V^{*}) \cong \bigoplus_{k \in I}
  \Hom_{\cat}(\unit , W_{k}).
  \]  
  But as $V$ is absolutely simple this $\Hom$-space is one-dimensional and
  hence there is precisely one $k \in I$ for which $\Hom_{\cat}(\unit ,
  W_{k})$ is non-zero.
  On the other hand, $p_{j} \coev_{V}: \unit \to W_{j}$ is necessarily
  nonzero.  Thus $k=j$.

The statements for $j'$ are argued similarly.
\end{proof}

\begin{lemma}\label{L:indecomp2} Let $W$ be an absolutely indecomposable
  object in $\cat$ and let $R$ be the radical of $\End_{\cat}(W)$.  If
  $\Hom_{\cat}(\unit , W)$ is one-dimensional, then
  \[
  r  f = 0
  \] 
  for all $r \in R$ and $f \in \Hom_{\cat}(\unit , W)$.  If $\Hom_{\cat}(W,
  \unit)$ is one-dimensional, then
  \[
  f  r = 0
  \] 
  for all $r \in R$ and
 $f \in \Hom_{\cat}( W, \unit)$.
\end{lemma}

\begin{proof}  
  Since by assumption $\Hom_{\cat}(\unit , W)$ is one-dimensional, we have
  that
  \[
  r  f = \gamma f
  \] 
  for some $\gamma \in \FK $.  By assumption elements of $R$ are nilpotent.
  Hence for $N \gg 0$ we have $r^{N}=0$ and
  \[
  0= r^{N}  f = \gamma^{N}f.
  \]  
  But $\FK$ is a field so this implies $\gamma=0$.

  The second statement is handled in an identical fashion.
\end{proof}

We are now prepared to
state and
 prove the main theorem.  For the reader's convenience, we recall the full set of assumptions in force at this point:   We assume $\kk$ is a field and $\cat$ is an additive pivotal \kt category.  We further assume all indecomposable objects are absolutely indecomposable and that the radical of the endomorphism ring of an absolutely indecomposable object consists of only nilpotent elements.

We also note the following identities which are needed in the proof:
\begin{gather}\label{E:Identities}
 \tcoev_{V^*}=(\phi_V\otimes \Id_{V^*})\coev_V=(\tev_V)^*, \\
 \ev_{V^*}= \tev_{V}(\phi_V^{-1}\otimes \Id_{V^*})=(\coev_V)^*. \notag
\end{gather}
\begin{theorem}\label{T:maintheorem}
  Let $V$ be an absolutely simple object in $\cat$ and let $j$ and $j'$ be as
  in Lemma~\ref{L:indecomp}.  Then the following conditions are equivalent:
  \begin{enumerate}
 \item the object $V$ is right ambi;
 \item $j=j'$;
 \item $W_j^*\cong W_j$.
 \end{enumerate}
\end{theorem}
\begin{proof}

The fact that the second and third conditions are equivalent is immediate from the fact that $V\otimes V^{*}$ is isomorphic to its dual, the Krull-Schmidt Theorem, Lemma~\ref{L:indecomp}, and the isomorphisms
\[
\Hom_{\cat}\left( W_{k}, \unit \right) \cong \Hom_{\cat}\left( \unit, W_{k}^{*}\right)
\]
and
\[
\Hom_{\cat}\left(\unit , W_{k} \right) \cong \Hom_{\cat}\left(W_{k}^{*}, \unit\right).
\]

We now show that the first and second conditions are equivalent. 
Assume $j=j'$.  Let $f \in \End_{\cat}\left(V\otimes V^{*} \right)$.  By linearity we may assume without loss that $f = \proj_{r}f\proj_{s}$ for some $r,s \in I$.
If either $r$ or $s$ is not equal to $j$, it then follows that both
$f\coev_{V}=0$ and 
$f^{*}\tcoev_{V^{*}} =0$.  Hence the ambidextrous condition is
trivially satisfied in this case.

Now assume $r=s=j$.  Now we consider $f':=p_{j}fi_{j} \in \End_{\cat}\left(W_{j} \right)$.  Since $W_{j}$ is absolutely indecomposable, we may write $f' = \alpha \Id_{W_{j}} + r$ for some $\alpha \in \kk$ and $r \in \operatorname{Rad}\left( \End_{\cat}\left(W_{j} \right)\right)$.  By Lemma~\ref{L:indecomp2} we then have 
\begin{equation*}
f' p_{j} \coev_{V} = \left(\alpha \Id_{W_{j}} + r \right)p_{j}\coev_{V} = \alpha p_{j}\coev_{V}.
\end{equation*}  Applying $i_{j}$ to both sides and simplifying yields 
\[
f \coev_{V} = \alpha\coev_{V}.
\]  

On the other hand, let us consider $\tev_{V} f$.  Using Lemma~\ref{L:indecomp2} and arguing as above we obtain 
\[
\tev_{V} f = \alpha \tev_{V}.
\]  Dualizing yields 
\[
f^{*}\left(\tev_{V} \right)^{*} = \alpha \left(\tev_{V} \right)^{*}
\]  Using \eqref{E:Identities} and applying $\phi_{V}^{-1}\otimes \Id_{V^{*}}$ to both sides then yields
\[
(\phi_{V}^{-1}\otimes \Id_{V^{*}}) f^{*}\tcoev_{V^{*}} = \alpha (\phi_{V}^{-1}\otimes \Id_{V^{*}}) \tcoev_{V^{*}}= \alpha \coev_{V}.
\]  Combining this with our calculation of $f\coev_{V}$ we have that $V$ is 
right 
ambidextrous.

On the other hand, say $V$ is 
right 
ambidextrous.  Then $\proj_{j}\coev_{V} = (\phi_{V}^{-1} \otimes  \Id_{V^{*}}) \proj_{j}^{*} \tcoev_{V^{*}}$.  By the choice of $j$ we have that $\proj_{j}\coev_{V} \neq 0$ and so 
\[
\proj_{j}^{*} \tcoev_{V^{*}}= (\phi_{V} \otimes  \Id_{V^{*}})\proj_{j}\coev_{V} \neq 0. 
\]
  Dualizing and using \eqref{E:Identities}, this then implies $\tev_{V} \proj_{j} \neq 0$.  However, by Lemma~\ref{L:indecomp} this implies $j=j'$.
\end{proof}

We note that identical arguments also prove the analogous statement for left
ambi objects.  It is also useful to note that the notion of right ambi is
local in the sense that it only depends on $\End_{\cat}(V \otimes V^{*})$.
For example, let $\cat$ be a pivotal $\kk$-category and let $\Fcat$ be a full
pivotal subcategory of $\cat$ (that is, the pivotal structure on $\Fcat$ is inherited from $\cat$).
Let $V$ be an absolutely simple object of $\Fcat$.
Then $V$ is right ambi in $\cat$ if and only if it is right ambi in $\Fcat$.

\subsection{The Ideal $\Proj$}\label{SS:proj}

We now assume $\cat$ is an abelian pivotal \kt category.  By \cite[Proposition 2.1.8]{BK} the tensor product in a pivotal category is exact in both entries and, hence, the full subcategory of $\cat$ consisting of the projective objects forms an ideal.   Let $\Proj$ denote this ideal.   Furthermore, since $\cat$ is a pivotal category it follows that if $P$ is a projective object, then $P^{*}$ is again a projective object (e.g.\ by \cite[Proposition 2.3]{EO}).  That is, projective and injective objects coincide in $\cat$.  Finally, we note that if $V$ is a projective object with $\tev_{V}: V \otimes V^{*} \to \unit$ an epimorphism, then $\ideal_{V}=\Proj$ by \cite[Lemma 12]{GPV}.

Let us now assume that $\cat$ has enough projectives.   Each absolutely simple object then has a projective cover and, since projectives are also injective, this projective cover has a unique absolutely simple subobject.  In particular, if we let $\Pu$ denote the projective cover of $\unit$, then we write $L$ for the unique absolutely simple subobject of $\Pu$. Following \cite{ENO}, we call $\cat$ \emph{unimodular} if 
\[
L \cong \unit.
\] That is, if $\Pu \cong \Pu^{*}$.

We recall the full set of assumptions currently in use:   We assume $\kk$ is a field and $\cat$ is an abelian pivotal \kt category.  We further assume all indecomposable objects are absolutely indecomposable and that the radical of the endomorphism ring of an absolutely indecomposable object consists of only nilpotent elements.  Finally, we assume $\cat$ has enough projectives. 

\begin{corollary}\label{C:PcoverC}  
   Let $\cat$ be a category which satisfies
  the above assumptions.  If $\cat$ contains absolutely simple projective objects, then every absolutely simple projective object is 
  right
  ambidextrous if and only if $\cat$ is unimodular.  
Therefore, if $\cat$ is unimodular and has an absolutely simple projective object, $L$, of $\cat$ such that $\tev_{L}$ is an epimorphism, then $\Proj$ admits a unique non-zero
 right trace.

\end{corollary}
\begin{proof}  
  If $L$ is an absolutely simple projective object, then $L\otimes L^{*}$
  decomposes into a direct sum of indecomposable projective objects.  By
  Theorem~\ref{T:maintheorem} and the definition of unimodularity, $L$ is
  ambidextrous if and only if $\cat$ is unimodular.  The second statement
of the corollary 
 is immediate from Theorem~\ref{T:idealtheorem} and the fact that $\ideal_{L}=\Proj$.  
\end{proof}

\section{Examples}\label{S:Examples}

\subsection{} We now apply Corollary~\ref{C:PcoverC} in a variety of settings.  In every case the proof involves verifying the two primary assumptions of Corollary~\ref{C:PcoverC}: the existence of a simple projective object and that the category is unimodular.  The remaining assumptions are well known in each case and left to the reader.

We note that in the following examples the categorical trace on the ideal $\Proj$ is known to be trivial.   Thus the nontrivial right trace given in the examples below cannot be the categorical trace.  We also remark that the categories in Sections~\ref{SS:CharZero}, \ref{SS:groups}, \ref{SS:Liealgebras}, \ref{SS:Liesuperalgebras}, and \ref{SS:Liesupalgebras} are all ribbon categories.  As remarked earlier, this implies the right trace on $\Proj$ is in fact a two-sided trace.

\subsection{Finite-dimensional Hopf algebras}\label{SS:Hopfalgebras}

Fix a ground field $\kk$ and let $H$ be a finite-dimensional Hopf algebra over $\kk$.  Let $\cat$ be $H$-mod, the category of finite-dimensional $H$-modules.  The counit, coproduct, and antipode define a unit object, a tensor product, and a duality on $\cat$.  If we write $S$ for the antipode of $H$, then by \cite[Proposition 2.1]{Bi} the category $\cat$ is a pivotal category if and only if there is a group-like element $g \in H$ such that $S^{2}(x)=gxg^{-1}$ for all $x \in H$.  That is, $\cat$ is a pivotal category if and only if $S^{2}$ is an inner automorphism.

Recall that $H$ is called \emph{unimodular} if the space of left and right integrals coincide (c.f.\ \cite[Section 2.1]{Mont}).   The following result of Lorenz \cite[Lemma 2.5]{Lor} shows that this definition agrees with the categorical notion given earlier. 

\begin{lemma}\label{L:selfduality} If $H$ is a finite-dimensional Hopf algebra over a field $\kk$, then the category of finite-dimensional $H$-modules is unimodular if and only if $H$ is unimodular. 
\end{lemma}

On the other hand, by Oberst and Schneider \cite{OS} (see Lorenz \cite[Proposition 2.5]{Lor}) a Hopf algebra $H$ is unimodular and $S^{2}$ is an inner automorphism if and only if $H$ is a symmetric algebra.

\subsection{Factorizable ribbon Hopf algebras in characteristic zero}\label{SS:CharZero}  

Let $H$ be a finite-dimensional, factorizable, ribbon Hopf algebra over an algebraically closed field of characteristic zero.  Such Hopf algebras appear frequently in the literature.  For example, if $H$ is the Drinfeld double of a finite-dimensional Hopf algebra, then it is well known to always be factorizable.  Furthermore, by Kauffman and Radford \cite{KR} is is also known precisely when the Drinfeld double is a ribbon Hopf algebra.  

\begin{theorem}\label{T:}   Let $H$ be a finite-dimensional, factorizable,
  ribbon Hopf algebra over an algebraically closed field of characteristic
  zero.  Let $\cat$ be the category of finite-dimensional $H$-modules.  Then
  the ideal $\Proj$ in $\cat$ admits a unique nontrivial trace.
\end{theorem}
\begin{proof}  Since $H$ is a ribbon Hopf algebra, $\cat$ is a ribbon category.  Since $H$ is assumed to be factorizable it follows that $\cat$ is 
unimodular, see for example
\cite[Proposition 4.5]{ENO}.  By work of Cohen and Westreich \cite{CW} the category $\cat$ always has at least one simple projective object.  
\end{proof}

\subsection{Group algebras and their quantum doubles}\label{SS:groups}
Now let us consider a fixed algebraically closed field $\kk$ of characteristic $p \geq 0$ and a finite group $G$. 
Then the group algebra $\kk G$ is a finite-dimensional Hopf algebra.  We also consider its quantum (or Drinfeld) double $D(G)=D(\kk G)$.  Although the braiding on $\kk G$-mod is symmetric, the braiding on $D(G)$-mod is generally not symmetric and even over the complex numbers one obtains nontrivial topological invariants \cite{TG}.  

Let $\cat$ denote the category of finite-dimensional $\kk G$-modules or $D(G)$-modules, as the case may be. For both $\kk G$ and $D(G)$ the category $\cat$ is an abelian pivotal \kt category. It is a classical fact that $\kk G$-mod is unimodular (e.g. see \cite[Theorem 6]{Alperin}).   On the other hand, combining the work of Radford \cite[Corollary 2 and Theorem 4]{Rad} and Farnsteiner \cite[Proposition 2.3]{Far} it is known that $D(G)$-mod is always unimodular.

If $p=0$ or $p>0$ and coprime to the order of $G$, then both $\kk G$-mod and $D(G)$-mod are semisimple categories.  In this case the only nonzero ideal is $\cat$ itself and the only nontrivial trace is the ordinary categorical trace.  Therefore in what follows we assume $p>0$ and divides the order of $G$. 

In \cite[Section 1]{With1} Witherspoon proved that $\kk G$-mod can be identified as a full subcategory of $D(G)$-mod and that projective objects in $\kk G$-mod are still projective in $D(G)$-mod under this identification.  Combining this with the above discussion we see that whenever $\kk G$-mod has a simple projective module Corollary~\ref{C:PcoverC} implies both $\kk G$-mod and $D(G)$-mod admits a nontrivial trace on the ideal $\Proj$.

To proceed we need a basic fact from the modular representation theory of finite groups.  By, for example, \cite[Corollary 6.3.4]{Ben} the existence of a simple projective $\kk G$-module is equivalent to showing that $\kk G$-mod has a block of defect zero.  The question of which finite groups have a block of defect zero goes back to Brauer and was settled for finite simple groups through the work of various people with the final case of the alternating groups handled by Granville and Ono \cite{GO}.  Summarizing this work we have the following result (see \cite{GO} for further details).

\begin{theorem}\label{T:defectzero}  Let $\kk$ be an algebraically closed field of characteristic $p$.   Assume $G$ is a group appearing in the classification of finite simple groups and:
\begin{enumerate}
\item  If $p=2$, then $G$ is not isomorphic to $M_{12}$, $M_{22}$, $M_{24}$, $J_{2}$, $HS$, $Suz$, $Ru$, $C1$, $C3$, $BM$, or $A_{n}$ where $n \neq 2m^{2}+m$ nor $2m^{2}+m+2$ for any integer $m$;
\item If $p=3$, then $G$ is not isomorphic to $Suz$, $C3$, or $A_{n}$ with $3n+1=m^{2}r$ where r is square-free and divisible by some prime $q \equiv 2$ mod $3$.
\end{enumerate} If $G$ and $p$ are as above, then $\kk G$ has a block of defect zero.
\end{theorem}

We remark that \cite{GO} also shows that the symmetric group on $n$ letters has a block of defect zero for all $p \geq 5$.  Applying this to our setting we obtain the following.

\begin{theorem}\label{T:Proj}  Let $\kk$ be an algebraically closed field of characteristic $p$ and let $G$ be a finite simple group listed in the previous theorem.  Or let $G$ be a symmetric group on $n$ letters and assume $p \geq 5$.  Then $\Proj$ admits a unique nontrivial trace in both $\kk G$-mod and $D(G)$-mod.
\end{theorem}

\subsection{Irrational conformal field theories}\label{SS:CFT} 
The investigation of conformal field theories leads to the study of pivotal
categories.  It was shown by Huang that the representation category of a
rational conformal vertex algebra is, in our language, a semisimple ribbon
category with finitely many simple objects (see \cite[Proposition 2]{Fuchs}
and references therein).  However, many problems in mathematics and physics
naturally lead to the study of irrational conformal field theories which need
not be semisimple.  As an example, consider the logarithmic tensor category
theory developed by Huang, Lepowsky, and Zhang (see \cite{HLZ} and its seven
sequels).  See also \cite{Fuchs} for a discussion of the non-semisimple case.
The following theorem suggests ambidextrous objects could be helpful in the
study of irrational conformal field theories.

Let $\cat$ denote the category $W(p)$-mod defined in \cite[Section 6]{Fuchs} and associated to the so-called $(1,p)$ minimal model.  Using the description of $\cat$ given in \cite{Fuchs,NT}, we see that $\cat$ is a pivotal category which contains a simple projective object and is unimodular.  Consequently we have the following theorem.

\begin{theorem}\label{T:CFT}  Let $\cat = W(p)$-mod be the category associated
  to the $(1,p)$ minimal model in \cite{Fuchs}.  Then the ideal $\Proj$ in
  $\cat$ admits a unique nontrivial right trace.
\end{theorem}

\subsection{Lie algebras in positive characteristic}\label{SS:Liealgebras}  Let $\kk$ be an algebraically closed field of characteristic $p>2$ and let $\fg$ be a restricted Lie algebra; that is, a Lie algebra defined over the field $\kk$ with an extra map $x \mapsto x^{[p]}$ giving $\fg$ a restricted structure.  For example, if $G$ is a reductive algebraic group defined over $\kk$, then $\fg = \operatorname{Lie}(G)$ is a restricted Lie algebra.   Note that for every $x \in \fg$ the element $x^{[p]}-x^{p}$ is necessarily central in the enveloping algebra $U(\fg )$.  Let $u(\fg )$ denote the restricted enveloping algebra 
of $\g$ 
defined by the quotient of 
 $U(\fg )$ by the ideal generated by all elements $x^{[p]}-x^{p}$ for $x \in \fg$.   Then $u(\fg )$ is a finite-dimensional Hopf algebra which inherits a Hopf algebra structure from $U(\fg )$.  Let $\Fcat_{0}$ denote the category of finite-dimensional $u(\fg)$-modules.  

Let $\mathcal{Z}$ denote the subalgebra of $U(\fg )$ generated by $x^{[p]}-x^{p}$ for $x \in \fg$.  Let $\Fcat$ be the category of all finite-dimensional $U(\fg )$-modules on which the elements of $\mathcal{Z}$ act semisimply.   Then $\Fcat$ decomposes into blocks according to the algebra homomorphisms $\mathcal{Z} \to \kk$.  In particular, $\Fcat_{0}$ is precisely the principal block of $\Fcat$ in this decomposition; that is, the full subcategory of all modules annihilated by $x^{[p]}-x^{p}$ for all $x \in \fg$. 

 The main result we need is due to Larson and Sweedler \cite[Corollary to Proposition 8]{LS} (see also \cite{Hum} for an alternate proof).
\begin{theorem}\label{T:Sweedler} If $\fg$ is a restricted Lie algebra, then $u(\fg )$ is unimodular if and only if the trace of $\operatorname{ad}(x)$ is zero for all $x \in \fg$, where $\operatorname{ad}$ denotes the adjoint representation.  In particular, if $\fg = \operatorname{Lie}(G)$ for some reductive algebraic group $G$, then $u(\fg )$ is unimodular.
\end{theorem}
We then have the following theorem.

\begin{theorem}\label{T:LieAlgebras}  Let $G$ be a reductive algebraic group defined over an algebraically closed field $\kk$ of characteristic $p>2$ and let $\fg = \operatorname{Lie}(G)$.  Then the ideal $\Proj$ in the categories $\Fcat_{0}$ and $\Fcat$ admits a 
unique 
nontrivial trace.
\end{theorem}

\begin{proof} Let $\St$ be the Steinberg representation for $\fg$.  Then $\St$ is a simple projective module for $u(\fg )$ \cite[Proposition II.10.2]{Ja}.  By Theorem~\ref{T:Sweedler} $\Fcat_{0}$ is unimodular.  Therefore by Corollary~\ref{C:PcoverC} $\St$ is ambidextrous and defines a nontrivial trace on $\Proj$ in $\Fcat_{0}$ and $\Fcat$.  This proves the desired result.
\end{proof}

We also note that if $\fg$-mod denotes the category of all finite-dimensional $\fg$-modules, then it is seen using results of \cite{Bend} that $\fg$-mod has no projective objects.  However, as $\Fcat_{0}$ is a full subcategory of $\fg$-mod, the Steinberg module still provides an ambi object in $\fg$-mod and so defines a trace on the ideal it generates.

The above theorem generalizes the explicit calculations done for $\fg = \mathfrak{sl}_{2}(\kk)$ in \cite{GKP}.  It is worth emphasizing that those calculations also show that there are simple modules for $\mathfrak{sl}_{2}(\kk)$  which are \emph{not} ambidextrous.

\subsection{Quantum groups at a root of unity}\label{SS:quantumgroups} In this subsection let $\fg$ be a semisimple complex Lie algebra.  Fix an odd integer $l>1$ which is coprime to three if $\fg$ contains a component of type $G_{2}$.  Fix $\zeta \in \C$ a primitive $l$th  root of unity.  Let $U_{\zeta}(\fg)$ (resp.\  $\mathcal{U}_{\zeta}(\fg )$) denote the restricted quantum group associated to $\fg$ obtained by specializing the Lusztig form (resp.\  the non-restricted quantum group associated to $\fg$ obtained by specializing the De Concini-Kac form) to the root of unity $\zeta$. Let $u_{\zeta}(\fg )$ denote the finite-dimensional Hopf algebra commonly known as the \emph{small quantum group}.  All three algebras are defined, for example, in \cite{CP}.

Let $\Fcat$ denote the category of all Type \textbf{1} finite-dimensional $U_{\zeta}(\fg )$-modules.  By Type \textbf{1} we mean that each central element $K_{i}^{l}$ acts by $1$.   There is no loss in assuming our modules are of Type \textbf{1} as the general case can easily be deduced from this one.  As we only consider representations in $\Fcat$, we replace $U_{\zeta}(\fg )$ with its quotient by the ideal generated by $K_{i}^{l}-1$.  Having done so, $u_{\zeta}(\fg )$ appears as the Hopf subalgebra of $U_{\zeta}(\fg )$ generated by $E_{i}, F_{i}, K^{\pm 1}_{i}$.  Thus we have a restriction functor from $\Fcat$ to $u_{\zeta}(\fg )$-mod, the category of finite-dimensional $u_{\zeta}(\fg )$-modules.  

On the other hand, $\mathcal{U}_{\zeta}(\fg )$ contains a large central Hopf subalgebra $\mathcal{Z}$ generated by $E_{i}^{l}$, $F_{i}^{l}$, and $K_{i}^{l}$.  Let $\Dcat$ denote the category of all finite-dimensional $\mathcal{U}_{\zeta}(\fg )$-modules on which the elements $K^{\pm 1}_{i}$ acts semisimply and the subalgebra $\mathcal{Z}$ also acts semisimply.  Then since $\mathcal{Z}$ acts semisimply, the category $\Dcat = \oplus_{\chi} \Dcat_{\chi}$, where the direct sum runs over all algebra homomorphisms $\chi: \mathcal{Z}\to \C$, and where $\Dcat_{\chi}$ is by definition the full subcategory of $\Dcat$ of all modules annihilated by $x-\chi(x)$ for all $x \in \mathcal{Z}$.

In particular, define $\chi_{0}$ by $\chi_{0}(E_{i}^{l})=\chi_{0}(F_{i}^{l})=0$ and $\chi_{0}(K_{i}^{l})=1$.  Then setting $K$ equal to the ideal of $\mathcal{U}_{\zeta}(\fg )$ generated by the kernel of $\chi_{0}$ we have $\mathcal{U}_{\zeta}(\fg )/K \cong u_{\zeta}(\fg)$, where the isomorphism is as Hopf algebras.  Hence there is an inflation functor from $u_{\zeta}(\fg )$-mod to $\Dcat$ which allows us to identify $u_{\zeta}(\fg )$-mod with $\Dcat_{\chi_{0}}$.  

\begin{theorem}\label{T:restrictedquantumgroup}  In each of the categories of
$u_{\zeta}(\fg )\text{-mod}$, $\Fcat$, and $\Dcat$, the
  ideal $\Proj$ admits a unique nontrivial right trace.
\end{theorem}

\begin{proof}  We first consider $u_{\zeta}(\fg )$-mod.  By \cite[Theorem 2.2]{Kum} $u_{\zeta}(\fg )$ is unimodular so the category $u_{\zeta}(\fg )$-mod is unimodular.  By \cite[Proposition 4.1]{Kum} $u_{\zeta}(\fg )$ has a simple projective object, $\St$, called the Steinberg representation. Thus by Corollary~\ref{C:PcoverC} the module $\St$ is ambidextrous and defines a unique nontrivial trace on $\Proj$ in $u_{\zeta}(\fg )$-mod.

We now consider the restricted quantum group $U_{\zeta}(\fg )$.  By \cite[Theorem 9.8]{APW} the Steinberg module $\St$ is in fact a simple projective object for $U_{\zeta}(\fg )$.  Since upon restriction to $u_{\zeta}(\fg )$ the module $\St$ is ambidextrous, it follows that $\St$ is ambidextrous in $\Fcat$, as well.  In particular, this implies $\Proj$ in $\Fcat$ admits a unique nontrivial trace.

Finally, let us consider the non-restricted quantum group $\mathcal{U}_{\zeta}(\fg )$.  Via the inflation functor, $\St$ defines a simple projective $\mathcal{U}_{\zeta}(\fg )$-module in $\Dcat_{\chi_{0}}$, hence in $\Dcat$.  Thus $\Proj$ in $\cat$ admits a unique nontrivial trace.
\end{proof}

Let us point out that this example has particular importance in low-dimensional topology.  The first and third author showed in \cite{GP} that the non-restricted quantum group $\mathcal{U}_{\zeta}(\fg )$ admits a nontrivial trace on the ideal $\Proj$ in $\Dcat$ by explicit calculations for certain ``typical'' modules.  As discussed in the introduction, this provides a number of 
results in the theory of 3-manifold invariants.   It is also worth noting that their calculations and \cite[Theorem 35]{GP} provided the original motivation for our Theorem~\ref{T:maintheorem}.

\subsection{Complex Lie superalgebras}\label{SS:Liesuperalgebras}  Let $\fg = \fg_{\0} \oplus \fg_{\1}$ be a finite-dimensional classical simple Lie superalgebra defined over the complex numbers.  By \emph{classical} we mean $\fg_{\0}$ is reductive as a Lie algebra.   That is, $\fg$ is a simple Lie superalgebra in the Kac classification \cite{Kac1} of type $ABCDPQ$, $D(2,1;\alpha)$, $F(4)$, or $G(3)$.   In fact what follows works equally well for their non-simple variants:  $\gl$, $\mathfrak{p}(n)$, $\fq(n)$, etc.  

We fix a Cartan subalgebra $\fh$ and Borel subalgebra $\fb$ of $\fg$ such that $\fh_{\0}$ and $\fb_{\0}$ define Cartan and Borel subalgebras of $\fg_{\0}$, respectively. Let $\Fcat$ be the category of all finite-dimensional $\fg$-supermodules which are completely reducible as $\fg_{\0}$-modules and all $\fg$-supermodule homomorphisms which preserve the $\Z_{2}$-grading.  In particular, $\Fcat$ is an abelian ribbon category which contains enough projectives, but nearly always fails to be semisimple. The simple objects of $\Fcat$ are classified up to parity change by their highest weight and we write $L(\lambda)$ for the simple supermodule of highest weight $\lambda \in \fh_{\0}^{*}$.

Note that if $\fg$ is a simple basic classical Lie superalgebra (ie.\ type $ABCD$, $D(2,1;\alpha), $ $F(4)$, or $G(3)$ in the Kac classification), then the typical simple supermodules are simple and projective \cite{Kac2}.  If $\fg$ is a Lie superalgebras of type $Q$, then it again is known to have typical representations which are simple and projective \cite[Lemma 4.51]{Br}. 

Viewing $\fg_{\1}$ as a $\fg_{\0}$-module via the adjoint action, let $\delta$ denote the one-dimensional $\fg_{\0}$-module given by   
\[
\delta = \Lambda^{\operatorname{dim}\fg_{\1}}\left(\fg_{\1} \right).
\]
   We call $\fg$ \emph{unimodular} if $\delta$ is the trivial $\fg_{\0}$-module.  This terminology is justified by the following lemma.  In what follows we abuse notation by writing $\delta \in \fh_{\0}^{*}$ for the weight of $\delta$.

Note that the proof of the lemma is an adaptation of an argument from \cite{MM}. In that paper they show quite a bit more than we need here; namely, they prove that most blocks of $\Fcat$ (and parabolic category $\mathcal{O}$) are symmetric categories.  In particular, their results answer many cases of a question raised in \cite{BKN3}. 

\begin{lemma}\label{L:unimodularLieSuperalgebra}  Let $\fg$ be a classical Lie superalgebra and let $\delta = \Lambda^{\operatorname{dim}\fg_{\1}}\left(\fg_{\1} \right)$.  Let $P_{0}$ be the projective cover in $\Fcat$ of the trivial supermodule.  Then the socle of $P_{0}$ is isomorphic to $L(\delta)$.  In particular, $\Fcat$ is unimodular if and only if $\fg$ is unimodular.
\end{lemma}

\begin{proof}  Let $P_{0}$ denote the projective cover in $\Fcat$ for the trivial module.  By \cite[Proposition 2.2.2]{BKN3} $P_{0}$ is the injective hull for some simple $\fg$-supermodule.  Say $L(\lambda)$ is the simple $\fg$-supermodule with $P_{0}\cong I(L(\lambda))$.  Note that the simple $\fg_{\0}$-module of highest weight $\lambda$, $L_{0}(\lambda)$, appears as a direct summand of $L(\lambda)$ as a $\fg_{\0}$-module.  Consider the $\fg$-supermodule
\[
M:=U(\fg ) \otimes_{U(\fg_{\0})} \left(L_{0}(\lambda) \otimes \delta^{*}  \right)\cong \Hom_{U(\fg_{\0})} \left( U(\fg ), L_{0}(\lambda)\right),
\] where the above isomorphism is as $\fg$-supermodules \cite[Proposition 2.2.1]{BKN3}.

Since restriction and $\Hom_{U(\fg_{\0})}(U(\fg), -)$ form an adjoint pair, we have 
\[
\Hom_{U(\fg )}(L(\lambda), M) \cong \Hom_{U(\fg_{\0})}(L(\lambda), L_{0}(\lambda)) \neq 0.
\] In particular this implies $I(L(\lambda)) \cong P_{0}$ appears as a direct summand of $M$.  Therefore, by Frobenius reciprocity, we have 
\[
0 \neq \Hom_{U(\fg )}(M,\C ) \cong \Hom_{U(\fg_{\0})}(L_{0}(\lambda) \otimes \delta^{*}, \C ) \cong \Hom_{U(\fg_{\0})}(L_{0}(\lambda), \delta).
\]  Thus $L_{0}(\lambda) \cong \delta$ as $\fg_{\0}$-modules.  In particular, $L(\lambda) \cong L(\delta)$.  

The statement on unimodularity now follows as an immediate corollary.

\end{proof}

The importance of $\delta$ was already observed in \cite{Go} where it was proven that all simple classical Lie superalgebras are unimodular.  We now come to the main theorem of the section. 

\begin{theorem}\label{T:ComplexLiesuperalgebras}  Let $\fg$ be a simple classical Lie superalgebra not of type $P$.  The ideal $\Proj$ in the category $\mathcal{F}$ admits a nonzero trace.
\end{theorem}

\begin{proof}  As mentioned above, by \cite[Section 3.4.3]{Go} we have that $\fg$ is unimodular for any simple classical Lie superalgebra and so by the previous lemma $\Fcat$ is unimodular.  As discussed above the so-called typical simple supermodules for the simple basic classical Lie superalgebras and for the type Q Lie superalgebras are simple and projective.  The result follows.
\end{proof}

We remark that in the case of type A and C the above result was first proven in \cite{GP2}.  There the authors give an explicit formula for the trace in terms of supercharacters. They also use deformation arguments to obtain ambidextrous objects for the Drinfeld-Jimbo quantum group over $\C[[h]]$ associated to $\fg$ which, in turn, can be used to define link invariants.

\subsection{Lie superalgebras in positive characteristic}\label{SS:Liesupalgebras}  Let $\kk$ be an algebraically closed field of characteristic $p$ and let $\fg$ be a restricted Lie superalgebra; that is, a Lie superalgebra defined over the field $\kk$ such that $\fg_{\0}$ is a restricted Lie algebra and $\fg_{\1}$ is a restricted $\fg_{\0}$-module via the adjoint action.  
In the following theorem we assume $\fg$ is one of the following restricted Lie superalgebras: $\gl$, $\fq (n)$, or a simple Lie superalgebra of type ABCD, $D(2,1;\alpha)$, $G(3)$, or $F(4)$.  We assume that the characteristic of the field $\kk$ is an odd prime and, in addition, greater than three if $\fg$ is of type $D(2,1;\alpha)$ or $G(3)$.  Let $\Fcat$ be the category of finite-dimensional $\fg$-supermodules on which the central elements $x^{p}-x^{[p]}$ ($x \in \fg_{\0}$) act semisimply.  We take as morphisms the $\Z_{2}$-grading preserving $\fg$-module homomorphisms.

\begin{theorem}  Let $\fg$ be as above.   Then the ideal $\Proj$ in $\Fcat$ admits a unique nontrivial trace.
\end{theorem}

\begin{proof}  If $\fg$ is $\fq(n)$, then by \cite[Proposition 2.1]{WZ1} it follows that the projective cover of the trivial module is self-dual.  In the other cases it follows by \cite[Proposition 2.7]{WZ2}.  Thus $\Fcat$ is unimodular.  If $\fg$ is $\fq (n)$, then by \cite[Theorem 3.10]{WZ1} we have projective simple objects in $\Fcat$.  In the other cases it follows by \cite[Theorem 4.7]{Z}.  
\end{proof}

\linespread{1}

\end{document}